\documentclass{article}
\usepackage[cp1251]{inputenc}
\usepackage[russian]{babel}
\usepackage{amssymb,amsfonts,amsmath,amsthm,amscd}
\usepackage{graphicx}
\usepackage{enumerate}
\usepackage{soul}
\usepackage[matrix,arrow,curve]{xy}
\usepackage{longtable}
\usepackage{array}
\usepackage{bm}

\newdimen\symskip
\newdimen\defskip
\newdimen\parind
\parind=\parindent
\newdimen\leftmarge
\newdimen\theoremshape
\topsep\defskip
\makeatletter
\newcommand*{\clei}{\nobreak\hskip\z@skip}

\renewcommand{\"}{''}
\renewcommand{\:}{\textup{:}}
\renewcommand{\~}{\textup{;}}
\DeclareRobustCommand*{\ti}{~\textemdash{} }
\DeclareRobustCommand*{\dh}{\clei\hbox{-}\clei}
\newcommand{\no}{}
\renewcommand{\@listI}{\settowidth\labelwidth{\labheadi{\no}}\listipar{\parind}{\labelwidth}}
\newcommand{\listivpar}{\topsep\defskip\partopsep0pt\parsep-\parskip\itemsep0.5\topsep}
\newcommand{\listipar}[2]{\rightmargin0pt\leftmargin#1\labelsep#1\advance\labelsep-#2\itemindent0pt\listivpar}
\renewcommand{\@listii}{\settowidth\labelwidth{\labheadii{\@roman{\no}}}\listiipar{\parind}{\labelwidth}}
\newcommand{\listiivpar}{\topsep0.5\defskip\partopsep0pt\parsep-\parskip\itemsep0.5\topsep}
\newcommand{\listiipar}[2]{\rightmargin0pt\leftmargin#1\labelsep#1\advance\labelsep-#2\itemindent0pt\listiivpar}
\def\thempfn{\ifcase\value{footnote}1\or *\or **\or ***\else\@ctrerr\fi}
\renewcommand\footnoterule{%
  \kern-3\p@
  \hrule\@width1in
  \kern2.6\p@}
\makeatother

\newcommand{\nwl}{\newline}

\makeatletter

\renewcommand{\@biblabel}[1]{[#1]}
\renewenvironment{thebibliography}[1]
     {\renewcommand{\refname}{References}%
      \section*{\refname}%
      \@mkboth{\MakeUppercase\refname}{\MakeUppercase\refname}%
      \list{\@biblabel{\@arabic\c@enumiv}}%
           {\itemsep\baselineskip
            \leftmargin\parind
            \settowidth\labelwidth{\@biblabel{#1}}%
            \labelsep\parind\advance\labelsep-\labelwidth
            \@openbib@code
            \usecounter{enumiv}%
            \let\p@enumiv\@empty
            \renewcommand\theenumiv{\@arabic\c@enumiv}}%
      \sloppy
      \clubpenalty4000
      \@clubpenalty\clubpenalty
      \widowpenalty4000%
      \sfcode`\.\@m}
     {\def\@noitemerr
       {\@latex@warning{Empty `thebibliography' environment}}%
      \endlist}

\def\@maketitle{%
  \newpage%\null
  \vskip0.5em%
  UDK \udk%
  \vskip0.5em%
  MSC \msc%
  \vskip1em%
  \begin{center}\bf%
  \let\footnote\thanks%
   {\Large\@author\par}%
   \vskip1.5em%
   {\LARGE\@title\par}%
   \vskip1em%
   {\large\@date}%
  \end{center}%
  \par
  \vskip1.5em}

\def\@title{\@latex@warning@no@line{No \noexpand\title given}}

\makeatother

\makeatletter

\renewcommand\sectionmark[1]{%
 \markright{%
  \ifnum \c@secnumdepth >\z@
   \thesection. \ %
  \fi
 #1}}%

\renewcommand{\section}{\@startsection{section}{1}{0pt}%
{5.5ex plus .5ex minus .2ex}{1.5ex plus .3ex}%
{\center\normalfont\Large\bfseries\sffamily\bom}}
\renewcommand{\subsection}{\@startsection{subsection}{2}{0pt}%
{4.5ex plus .4ex minus .2ex}{0.75ex plus .2ex}%
{\center\normalfont\large\bfseries\sffamily\bom}}
\renewcommand{\subsubsection}{\@startsection{subsubsection}{3}{0pt}%
{2.5ex plus .5ex minus .2ex}{1ex plus .2ex}%
{\center\normalfont\bfseries\sffamily\bom}}

\newcommand{\Ss}{\textup{\S\,}}

\def\@postskip@{\hskip.5em\relax}

\def\postsection{.\@postskip@}

\def\postsubsection{.\@postskip@}

\def\postsubsubsection{.\@postskip@}

\def\postparagraph{.\@postskip@}

\def\postsubparagraph{.\@postskip@}

\def\@seccntformat#1{\csname pre#1\endcsname\csname the#1\endcsname\csname post#1\endcsname}

\makeatother

\renewcommand{\thesection}{\textup{\arabic{section}}}

\newcommand{\parr}{\par\addvspace{\defskip}}
\newcommand{\theo}[2]{\newtheorem{#1}{#2}[section]}
\newcommand{\deff}[2]{\newenvironment{#1}{\parr\textbf{#2.}}{\parr}}
\theo{cas}{Case}
\theo{problem}{Problem}
\theo{theorem}{Theorem}
\theo{lemma}{Lemma}
\theo{prop}{Proposition}
\theo{stm}{Statement}
\theo{imp}{Corollary}
\theo{ex}{Example}
\deff{df}{Definition}
\deff{note}{Remark}
\deff{denote}{Notation}
\deff{denotes}{Notations}
\deff{hint}{Hint}
\deff{answer}{Answer\:}

\makeatletter
\def\@begintheorem#1#2[#3]{%
  \deferred@thm@head{\the\thm@headfont \thm@indent
    \@ifempty{#1}{\let\thmname\@gobble}{\let\thmname\@iden}%
    \@ifempty{#2}{\let\thmnumber\@gobble}{\let\thmnumber\@iden}%
    \@ifempty{#3}{\let\thmnote\@gobble}{\let\thmnote\@iden}%
    \thm@notefont{\bfseries\upshape}%
    \indent%
    \thm@swap\swappedhead\thmhead{#1}{#2}{#3}%
    \the\thm@headpunct
    \thmheadnl % possibly a newline.
    \hskip\thm@headsep
  }%
  \ignorespaces}
\renewenvironment{proof}{\setcounter{cas}{0}\parr\pushQED{\qed}\normalfont$\square\quad$}{\setcounter{cas}{0}\popQED\@endpefalse\parr}

\makeatother

\newcommand{\labheadi}[1]{\textup{#1)}}
\newcommand{\labheadii}[1]{\textup{(#1)}}

\newenvironment{nums}[1]{\renewcommand{\no}{#1}\begin{enumerate}}{\end{enumerate}}
\newcommand{\eqn}[1]{\begin{equation}#1\end{equation}}
\newcommand{\equ}[1]{\begin{equation*}#1\end{equation*}}

\newcommand{\case}[1]{\begin{cases}#1\end{cases}}
\newcommand{\cask}[1]{\begin{casks}#1\end{casks}}

\newcommand{\rbmat}[1]{\begin{pmatrix}#1\end{pmatrix}}

\makeatletter

\def\LT@makecaption#1#2#3{%
  \LT@mcol\LT@cols c{\hbox to\z@{\hss\parbox[t]\LTcapwidth{%
    \sbox\@tempboxa{#1{#2. }#3}%
    \ifdim\wd\@tempboxa>\hsize
      #1{#2. }#3%
    \else
      \hbox to\hsize{\hfil\box\@tempboxa\hfil}%
    \fi
    \endgraf\vskip\baselineskip}%
  \hss}}}

\@addtoreset{equation}{section}
\@addtoreset{footnote}{section}
\@addtoreset{footnote}{page}

\newenvironment{casks}{%
  \matrix@check\casks\env@casks
}{%
  \endarray\right.%
}
\def\env@casks{%
  \let\@ifnextchar\new@ifnextchar
  \left\lbrack
  \def\arraystretch{1.2}%
  \array{@{}l@{\quad}l@{}}%
}

\newcommand{\high}{\vph{\Br{}_0^0}}
\newcounter{numt}
\newcounter{col}
\newcounter{coll}
\newcommand{\mt}[3]{\multicolumn{#1}{#2}{#3}}
\renewcommand{\thenumt}{\textup{\arabic{numt})}}
\newcommand{\news}{\\\hline}
\newcommand{\refs}{\refstepcounter{numt}}
\newcommand{\an}[1]{\refs\label{#1}\thenumt&}
\newcommand{\n}[1]{\news\an{#1}}
\newcommand{\refcol}[1]{\addtocounter{col}{#1}}
\newcommand{\fmc}{\news\mt{1}{|>{\nwl}p{0.5cm}<{\nwl}|}{$\high$ №}}
\newcommand{\mc}[3]{&\mt{#1}{#2@{\refcol{#1}}|}{#3}}
\newcommand{\nc}[2]{\mc{1}{#1}{#2}}
\newcommand{\mtc}[1]{\mt{\value{coll}}{c}{#1}}

\newcommand{\lontu}[4]{%
\setcounter{col}{1}\setcounter{numt}{0}
\begin{longtable}{|>{\nwl}p{0.5cm}<{\nwl}|#1}\caption*{#2}\fmc#3
\setcounter{coll}{\value{col}}\news\endfirsthead
\fmc#3\news\endhead
\mtc{}\\
\mtc{\textit{Continuation on the next page}}
\endfoot
\endlastfoot
#4\news\end{longtable}}

\makeatother

\renewcommand{\ge}{\geqslant}
\renewcommand{\le}{\leqslant}
\newcommand{\notni}{\not\ni}
\newcommand{\fa}{\,\forall\,}
\newcommand{\exi}{\,\exists\,}

\newcommand{\bes}{\infty}
\newcommand{\es}{\varnothing}

\newcommand{\subs}{\subset}

\newcommand{\sm}{\setminus}

\newcommand{\cln}{\colon}
\newcommand{\nl}{\lhd}

\newcommand{\Ra}{\Rightarrow}
\newcommand{\Lra}{\Leftrightarrow}

\newcommand{\hra}{\hookrightarrow}
\newcommand{\dv}{\smash{\mskip3mu\lower1pt\hbox{\vdots}\mskip3mu}}

\newcommand{\wt}{\widetilde}
\newcommand{\wh}{\widehat}

\newcommand{\sums}[1]{\sum\limits_{{#1}}}

\renewcommand{\caps}[1]{\bigcap\limits_{{#1}}}

\newcommand*{\bw}[1]{#1\nobreak\discretionary{}{\hbox{$\mathsurround=0pt #1$}}{}}
\newcommand{\sco}{,\ldots,}

\newcommand{\sti}{\bw\times\ldots\bw\times}

\newcommand{\und}[1]{\underbrace{#1}}

\newcommand{\br}[1]{\bigl(#1\bigr)}
\newcommand{\Br}[1]{\Bigl(#1\Bigr)}

\newcommand{\ter}[1]{\textup{(}#1\textup{)}}

\newcommand{\bc}[1]{\bigl\{#1\bigr\}}

\newcommand{\mbb}{\mathbb}

\newcommand{\N}{\mbb{N}}

\newcommand{\Ga}{\Gamma}
\newcommand{\de}{\delta}

\newcommand{\si}{\sigma}

\DeclareMathOperator{\Ann}{Ann}

\DeclareMathOperator{\rad}{rad}
\DeclareMathOperator{\diam}{diam}

\newcommand{\fm}[1]{\footnote{#1}}

\newcommand{\fnmark}{\footnotemark}
\newcommand{\fntext}[1]{\footnotetext{#1}}
\newcommand{\fon}[2]{\renewcommand{\thefootnote}{#1}\fnmark\fntext{#2}}
\newcommand{\fnmar}[1]{\renewcommand{\thefootnote}{#1}\fnmark}

\newcommand{\bom}{\boldmath}

\newcommand{\vph}[1]{\vphantom{#1}}

\begin{document}

\author{O.\,G.\?Styrt}
\title{Orthogonality graphs of matrices\\
over commutative rings}
\date{}
\newcommand{\udk}{512.552+ 512.643+519.173.1+519.173.5}
\newcommand{\msc}{13A70+05C12+05C25+05C40}

\maketitle

{\leftskip\parind\rightskip\parind
The paper is devoted to studying the orthogonality graph of the matrix ring over a~commutative ring. It is proved that the orthogonality graph of the
ring of matrices with size greater than~$1$ over a~commutative ring with zero-divisors is connected and has diameter $3$ or~$4$\~ a~criterion for each
value is obtained. It is also shown that each of its vertices has distance at most~$2$ from some scalar matrix.

\smallskip

\textbf{Key words\:} associative ring with identity, commutative ring, zero-divisor, matrix ring, zero-divisor graph, orthogonality graph.\par}

\section{Introduction}\label{introd}

Researching properties of associative rings in terms of graphs of some naturally occurring algebraic binary relations takes an important place in modern
mathematics. Thus, \textit{a~zero-divisor graph} was first defined in 1986 by Beck~\cite{Beck} for a~commutative ring. Its vertices were all
zero-divisors, and edges connected exactly all pairs of distinct elements giving zero in product. But since 1999 one uses its more convenient
interpretation introduced by Anderson and Livingston in~\cite{AL1} via excluding the zero element of the ring from its vertex set. It is also proved
in~\cite{AL1} that the zero-divisor graph of a~commutative ring is connected and has diameter at most three\~ in the former treatment of the graph these
statements would be trivial. A~number of further papers also studies various characteristics of the zero-divisor graph\: center and radius~\cite{Rd},
concepts of planarity~\cite{AMY} and uniqueness of determining the ring by the graph up to an isomorphism~\cite{AL2,AM1}. For non-commutative rings,
there are several types of graphs defined by zero-divisors\:
\lontu{ >{\nwl}p{2cm}<{\nwl} | >{\nwl}p{2cm}<{\nwl} | >{\nwl}p{3cm}<{\nwl} | >{\nwl $}p{2cm}<{$ \nwl} | >{\nwl}p{1cm}<{\nwl} |}
{}
{\nc{>{\nwl}p{2cm}<{\nwl}}{Name} \nc{>{\nwl}p{2cm}<{\nwl}}{Edge orientation} \nc{>{\nwl}p{3cm}<{\nwl}}{Vertices}
\nc{>{\nwl}p{2cm}<{\nwl}}{Edge\nwl from~$x$ to~$y$} \nc{>{\nwl}p{1cm}<{\nwl}}{See}}{%
\an{or}
Directed zero-divisor graph & Yes & One- and two-sided\nwl zero-divisors & xy=0 & \cite{AM2,AM3}
\n{neor}
(Undirected) zero-divisor graph & No & Nonzero\nwl one- and two-sided\nwl zero-divisors & \cask{xy=0\\yx=0} & \cite{AM3}
\n{ort}
Orthogonality graph & No & Nonzero two-sided\nwl zero-divisors & \case{xy=0\\yx=0} & \cite{GM1,GM2}}

The main results for orthogonality graphs of non-commutative rings found by now concern primarily matrix rings. Thus, in the case of the basic ring being
a~skew field, the following properties of the orthogonality graph of the $(n\times n)$\dh matrix ring are obtained: once $n=2$, it is
disconnected and all its connected components have diameters at most~$2$, and, once $n\ge3$, it is connected and has diameter~$4$. These statements are
proved in 2014 for a~field~\cite{GM1} and later, in 2017\ti for an arbitrary skew field~\cite{GM2}\~ they can also be easily generalized to integral
domains (by reducing to the field of fractions).

In this paper, there will be the orthogonality graph of the matrix ring over a~commutative ring with zero-divisors studied and the following main result
proved.

\begin{theorem}\label{main} Let $R$ be a~commutative ring with zero-divisor set $Z_R\ne\{0\}$. Then, for any $n>1$, the orthogonality graph of the ring
of $(n\times n)$\dh matrices over~$R$ is connected and has diameter $3$ or~$4$, the value~$3$ being equivalent to the relation
\eqn{\label{crit}
\fa a_0\in Z_R\quad \exi a_1,a_2\in R\sm\{0\}\quad\quad \fa i,j\in\{0,1,2\},\ i\ne j\cln\quad a_ia_j=0,}
and each of its vertices has distance at most~$2$ from some scalar matrix.
\end{theorem}

\begin{theorem}\label{submain} Let $r$ be the radius of the graph under conditions of Theorem~\ref{main}. Then
\begin{nums}{-1}
\item\label{rad4} $2\le r\le4$\~
\item\label{rad3} if \eqref{crit} holds, then $r\in\{2;3\}$\~
\item\label{rad2} $r=2$ if and only if there exists an element $c\in R\sm\{0\}$ such that
\eqn{\label{cond}
\fa a\in Z_R\quad \Ann(c)\cap\Ann(a)\ne0.}
\end{nums}
\end{theorem}

\section{Auxiliary agreements}

In the paper, the following notations and agreements will be used.

\begin{nums}{-1}

\item Set-theoretical\:
\begin{itemize}
\item While listing elements of a~disordered set, figured brackets are used. As for elements of an ordered tuple, they are listed in round brackets and
can be repeated.
\item $D^n:=\und{D\sti D}_n$ is the $n$\dh ary Cartesian power of a~set~$D$.
\end{itemize}

\item General algebraic\:
\begin{itemize}
\item All rings considered are supposed to be associative and with identity.
\item $R$ is an arbitrary ring.
\item For any subset $D\subs R$, define $D^*:=D\sm\{0\}$. In particular, by~$R^*$ denote the subset of all nonzero (not necessarily invertible as in
standard interpretation) elements of~$R$.
\item An ideal in~$R$ is \textit{proper} if it does not equal~$R$.
\item $M_{m\times n}(R)$ is the $R$\dh module of $(m\times n)$\dh matrices over~$R$\~ $M_n(R)$ is the ring $M_{n\times n}(R)$. If in the brackets the
ring is replaced with some of its subsets~$D$, then the subset of all matrices with entries from~$D$ is meant.
\item $0^m_n$ is the zero $(m\times n)$\dh matrix\~ $0_n:=0^n_n$\~ $E_n$ is the identity $(n\times n)$\dh matrix\~ $J_r$ is the Jordan cell of size~$r$
with eigenvalue~$0$. If the matrix sizes are clear from the context, then the indices can be omitted.
\item $E_{kl}$ is the matrix unit $(a_{ij})$, $a_{ij}:=\de_{ki}\de_{lj}$.
\item For a~square matrix~$A$ over a~commutative ring\: $\wt{A}$ is its cofactor matrix\~ $\wh{A}:=\br{\wt{A}}^T$.
\item If $A=(a_{k_1,k_2})\in M_{n_1\times n_2}(R)$, $P_i\in\{1\sco n_i\}^{m_i}$ ($i=1,2$), then $A^{P_1}_{P_2}$ is the matrix
$(b_{l_1,l_2})\in M_{m_1\times m_2}(R)$, $b_{l_1,l_2}:=a_{k_1(l_1),k_2(l_2)}$, where $k_i(l_i)$ is the $l_i$\dh th element of~$P_i$. If numbers are
repeated neither in~$P_1$, nor in~$P_2$, then $A^{P_1}_{P_2}$ is the submatrix of~$A$ with row and column numbers from $P_1$ and~$P_2$ respectively.
\end{itemize}

\item On zero-divisor types\:
\begin{itemize}
\item An element $a\in R$ is called
\begin{itemize}
\item a~\textit{left \ter{\textup{resp.} right} zero-divisor} if there exists an element $b\in R^*$ such that $ab=0$ (resp. $ba=0$)\~
\item a~\textit{zero-divisor} if it is either left or right zero-divisor\~
\item a~\textit{two-sided zero-divisor} if it is both left and right zero-divisor.
\end{itemize}
At that,
\begin{itemize}
\item in a~commutative ring, the concepts of all zero-divisor types are equivalent\~
\item zero is a~two-sided zero-divisor\~ if there are no other zero-divisors, then $R$ is called a~\textit{ring without zero-divisors}.
\end{itemize}
\item An \textit{integral domain} is a~commutative ring without zero-divisors.
\end{itemize}

\item From general graph theory\:
\begin{itemize}
\item All graphs considered are assumed to be undirected.
\item $\Ga=(V,E)$ is an arbitrary graph\~ $V$ and~$E$ are its vertex and edge sets respectively. In doing so, one can (usually with more convenience)
define~$E$ via a~symmetric binary relation on~$V$.
\item Two vertices are \textit{adjacent} if they are connected with an edge.
\item A~\textit{subgraph} is a~graph with vertex set $V'\subs V$ and, unless otherwise stated, with the same binary relation restricted on~$V'$.
\item A~\textit{path} is a~sequence of vertices where any two neighbor ones are adjacent.
\item The \textit{length} of a~path is the number of its edges.
\item The \textit{distance} between vertices $v$ and~$w$ (not. $d(v,w)$) is the minimum of lengths of paths between them\~ if they do not exist, then set
$d(v,w):=+\bes$\~ the sign is obvious in this context and therefore will be omitted. Clearly, $\br{d(v,w)=0}\Lra(v=w)$.
\item The \textit{distance} from a~vertex~$v$ to a~subset $W\subs V$ (not. $d(v,W)$) is the number\fon{*}{Possibly $\bes$.}
\equ{
\min\bc{d(v,w)\cln w\in W}.}
\item $d(v):=\sup\bc{d(v,w)\cln w\in W}$ ($v\in V$).
\item The \textit{diameter} of~$\Ga$ is the number\fnmar{*}\
\equ{
\diam(\Ga):=\sup\bc{d(v,w)\cln v,w\in V}=\max\bc{d(v)\cln v\in W}.}
\item \item The \textit{radius} of~$\Ga$ is the number\fnmar{*}\ $\rad(\Ga):=\min\bc{d(v)\cln v\in W}$. Clearly,
\eqn{\label{radi}
\rad(\Ga)\le\diam(\Ga)\le2\cdot\rad(\Ga).}
\item A~graph is \textit{connected} if there exists a~path between any two of its vertices.
\begin{note} It is easy to see that a~graph with finite diameter is connected. The converse fails\~ an example is the set of positive integers with the
neighborhood relation.
\end{note}
\end{itemize}

\item On special graphs in algebraic structures\:
\begin{itemize}
\item $O(R)$ is the orthogonality graph of the ring~$R$ (for a~commutative ring it is the same as the zero-divisor graph).
\item Vertices of $O(R)$ are all nonzero two-sided zero-divisors of~$R$\~ the orthogonality relation ($xy=yx=0$) is written as ($x\perp y$)\~ $O_R(x)$ is
the set of all vertices orthogonal to~$x$.
\end{itemize}

\end{nums}

\section{Proofs of the results}

Consider an arbitrary commutative ring~$R$. Denote by $\Ann(a)$ ($a\in R$) the ideal $\{x\bw\in R\cln ax=0\}$ and by~$Z_R$ the set
$\bc{a\in R\cln\Ann(a)\ne0}$ of all zero-divisors. Further, let $S$ be the ring $M_n(R)$ ($n>1$). Via the natural ring embedding $R\hra S,\,a\to aE$,
identify~$R$ with the subring $RE\subs S$ (and, thus, $O(R)$\ti with a~subgraph of the graph $O(S)$). For $A\in S$, set $I_A:=\Ann(\det A)\nl R$.

The graph $O(R)$ is connected and has diameter at most~$3$ (see Theorem~2.3 in~\cite[\Ss2]{AL1}). Besides, if $R$ is a~skew body, then
\begin{nums}{-1}
\item once $n=2$, the graph $O(S)$ is disconnected and all its connected components have diameters $\le2$\~
\item once $n\ge3$, the graph $O(S)$ is connected and has diameter~$4$.
\end{nums}
These results are obtained in~\cite[\Ss4]{GM1} for fields (Lemma~4.1 and Theorem~4.5 respectively), and in~\cite[\Ss2]{GM2} are generalized to arbitrary
skew-fields (Lemma~2.2 and Theorem~2.1 respectively). They are also shifted to integral domains (by reducing to the field of fractions).

\begin{theorem}\label{ide} For any matrix $A\in S$ and proper ideal $I\nl R$ containing $\det A$, there exists a~matrix $B\in S\sm\br{M_n(I)}$ such that
$AB,BA\in M_n(I)$.
\end{theorem}

\begin{proof} For $m\in\N$, set $Q_m:=\{1\sco m\}$ and $P_m:=(1\sco m)\in\N^m$.

Consider all triples $(k,P',P\")$ ($k\ge0$, $P',P\"\in(Q_n)^k$) satisfying the relation $\det(A^{P'}_{P\"})\bw\notin I$. For each of them, numbers are
repeated neither in~$P_1$, nor in~$P_2$, and, by condition, $k<n$. Besides, at least one of such triples exists\: for $k:=0$ and empty tuples $P',P\"$,
the corresponding $(0\times 0)$\dh matrix has determinant $1\notin I$. Hence, we can fix one of these triples with the largest possible~$k$, and then
$0\le k<n$, $m:=k+1\in Q_n$.

\begin{cas}\label{fir} $P'=P\"=P_k$.
\end{cas}

By construction, $\det(A^{P_k}_{P_k})\notin I$. Further, set $C:=A^{P_m}_{P_m}\in M_m(R)$,
\equ{
B:=\rbmat{\wh{C} & 0^m_{n-m}\\0^{n-m}_m & 0_{n-m}}\in S.}
Then $b_{m,m}=\det(A^{P_k}_{P_k})\notin I$ implying $B\notin M_n(I)$. Show that $AB,A^TB^T\in M_n(I)$, i.\,e. that, for any $p,q\in Q_n$, the matrix
entries $(AB)_{p,q}$ and $(A^TB^T)_{p,q}$ belong to~$I$. Assume that $p\in Q_n$ and $q\in Q_m$ (otherwise $(AB)_{p,q}=(A^TB^T)_{p,q}=0$). Let
$P\in(Q_n)^m$ be the tuple obtained form~$P_m$ by changing the $q$\dh th element with~$p$. Due to maximality of~$k$ and the inequality $m>k$, we have
$\det(A^P_{P_m}),\det(A^{P_m}_P)\in I$,
\equ{\begin{aligned}
(AB)_{p,q}=\sums{i\in Q_n}(a_{p,i}b_{i,q})=\sums{i\in Q_m}\br{a_{p,i}(\wh{C})_{i,q}}=&\sums{i\in Q_m}\br{a_{p,i}(\wt{C})_{q,i}}=\det(A^P_{P_m})\in I;\\
(A^TB^T)_{p,q}=\sums{i\in Q_n}\br{(A^T)_{p,i}(B^T)_{i,q}}=&\sums{i\in Q_m}\br{a_{i,p}(\wt{C})_{i,q}}=\det(A^{P_m}_P)\in I.
\end{aligned}}
Thereby, it is proved that $AB,(BA)^T=A^TB^T\in M_n(I)$ implying $BA\in M_n(I)$.

\begin{cas}\label{arb} $P',P\"\in(Q_n)^k$ are arbitrary tuples.
\end{cas}

In each of the tuples $P'$ and~$P\"$ all numbers are distinct. Hence, via suitable permutations of rows and columns, one can obtain from~$A$
a~matrix~$A_0$ satisfying Case~\ref{fir} with the same~$k$. By proved above, there exists a~matrix $B_0\in S\sm\br{M_n(I)}$ such that
$A_0B_0,B_0A_0\bw\in M_n(I)$. At that, there exist monomial (therefore, invertible) matrices $C_1,C_2\in S$ such that $A\bw=C_1A_0C_2^{-1}$.
Left (resp. right) multiplying a~matrix by a~monomial one permutes its rows (resp. columns), and, consequently, $B:=C_2B_0C_1^{-1}\in S\sm\br{M_n(I)}$,
$AB\bw=C_1(A_0B_0)C_1^{-1}\bw\in M_n(I)$, $BA=C_2(B_0A_0)C_2^{-1}\in M_n(I)$.
\end{proof}

\begin{imp}\label{orde} If $A\in S$ and $c\in I_A^*$, then, in the subset $(cS)^*\subs S$, there exists an element orthogonal to~$A$.
\end{imp}

\begin{proof} By condition, $I:=\Ann(c)\nl R$ is a~proper ideal containing $\det A$. According to Theorem~\ref{ide}, there exists a~matrix
$B\in S\sm\br{M_n(I)}$ such that $AB,BA\in M_n(I)$. Thus, $C:=cB\ne0$ and $c(AB)=c(BA)=0$, i.\,e. $C\in(cS)^*$ and $AC=CA=0$.
\end{proof}

\begin{lemma}\label{eqi} For any $A\in S$, the following conditions are equivalent\:
\begin{nums}{-1}
\item\label{dedi} $\det A\in Z_R$\~
\item\label{ine} $I_A\ne0$\~
\item\label{orto} in~$S^*$, there exists an element orthogonal to~$A$\~
\item\label{zdi2} $A$ is a~two-sided zero-divisor\~
\item\label{zdi1} $A$ is a~zero-divisor.
\end{nums}
\end{lemma}

\begin{proof} The implications $\text{\ref{dedi}}\Lra\text{\ref{ine}}$ and $\text{\ref{orto}}\Ra\text{\ref{zdi2}}\Ra\text{\ref{zdi1}}$ obviously follow
from definitions, and the implication $\text{\ref{ine}}\Ra\text{\ref{orto}}$\ti from Corollary~\ref{orde}.

Prove the implication $\text{\ref{zdi1}}\Ra\text{\ref{dedi}}$. Suppose that, without loss of generality, $A$ is a~left zero-divisor, i.\,e. that $AB=0$
for some $B\in S^*$. Then $\wh{A}A=(\det A)E$ implying $(\det A)B\bw=\wh{A}AB=0$. It remains to use non-triviality of~$B$.
\end{proof}

\begin{imp} All zero-divisors in~$S$ are two-sided.
\end{imp}

Let $Z_S\subs S$ be the subset of all elements $A\in S$ satisfying each of the equivalent conditions \ref{dedi}---\ref{zdi1} of Lemma~\ref{eqi},
i.\,e. the set of all zero-divisors of the ring~$S$. Then the vertex set of the graph $O(S)$ is~$Z_S^*$.

Further, we will assume that $Z_R^*\ne\es$.

\begin{stm}\label{bas} If $I\nl R$ and $I\ne0$, then $Z_R\cap I\ne\{0\}$.
\end{stm}

\begin{proof} Suppose that $Z_R\cap I=\{0\}$. There exist elements $b\in I^*$ and $c\in Z_R^*$\~ then $bc\bw\in Z_R\cap I=\{0\}$. So, $bc=0\ne c$ that
implies $b\in Z_R\cap I^*=\es$, a~contradiction.
\end{proof}

\begin{lemma}\label{fam} If, for a~subset $D\subs S$, the ideal $I:=\caps{A\in D}I_A\nl R$ is nonzero, then there exist elements $b\in Z_R^*$ and
$C_A\in S^*$, $A\in D$, such that $bE\perp C_A\perp A$ \ter{$A\in D$}.
\end{lemma}

\begin{proof} According to Statement~\ref{bas}, the ideal~$I$ contains an element $c\in Z_R^*$. Then $bc=0$ where $b\in Z_R^*$. Further, for any
$A\in D$, we have $c\in I_A^*$ and, by Corollary~\ref{orde}, there exist an element $C_A\in(cS)^*$ orthogonal to~$A$\~ at that, $bC_A\in bcS=0$,
$bE\perp C_A$.
\end{proof}

\begin{imp}\label{di24}\
\begin{nums}{-1}
\item For any $A\in Z_S^*$, we have $d\br{A,O(R)}\le2$.
\item If $A_1,A_2\in Z_S^*$ and $I_{A_1}\cap I_{A_2}\ne0$, then $d(A_1,A_2)\le4$.
\end{nums}
\end{imp}

\begin{proof} It suffices to apply Lemma~\ref{fam} to the subsets $\{A\},\{A_1,A_2\}\subs S$.
\end{proof}

\begin{lemma}\label{di3} If $A_i\in Z_S^*$, $c_i\in I_{A_i}^*$ \ter{$i=1,2$} and $c_1c_2=0$, then $d(A_1,A_2)\le3$.
\end{lemma}

\begin{proof} By Corollary~\ref{orde}, for each $i=1,2$, there exists an element $C_i\in(c_iS)^*$ such that $C_i\perp A_i$. In this case,
$C_1C_2,C_2C_1\in c_1c_2S=0$, $C_1\perp C_2$.
\end{proof}

\begin{df} We will say that an ideal $I\nl R$ \textit{does not have zero-divisors} if $I^*I^*\notni0$, i.\,e. if the \textit{ring}\fm{In general, without
identity.}~$I$ does not have zero-divisors.
\end{df}

\begin{lemma}\label{ieq} If $A_1,A_2\in Z_S^*$ and $d(A_1,A_2)>3$, then $I_{A_i}$ \ter{$i=1,2$} is the same ideal without zero-divisors.
\end{lemma}

\begin{proof} According to Lemma~\ref{di3}, $I_{A_1}^*I_{A_2}^*\notni0$. It remains to prove that $I_{A_1}=I_{A_2}$.

Suppose that $I_{A_1}\ne I_{A_2}$. Without loss of generality, assume that there exists an element $c\in I_{A_1}\sm I_{A_2}$. Setting $a:=\det A_2$,
we have $I_{A_2}=\Ann(a)$, $b:=ca\in I_{A_1}^*$ and $bI_{A_2}=caI_{A_2}=0$ implying $bI_{A_2}^*\subs\{0\}\cap(I_{A_1}^*I_{A_2}^*)=\es$, $I_{A_2}^*=\es$,
$I_{A_2}=0$, a~contradiction.
\end{proof}

\begin{theorem}\label{dia4} The graph $O(S)$ is connected and has diameter at most~$4$.
\end{theorem}

\begin{proof} Suppose that there exist elements $A_1,A_2\in Z_S^*$ satisfying the inequality $d(A_1,A_2)\bw>4$. By Lemma~\ref{ieq},
$0\ne I_{A_1}=I_{A_2}=I_{A_1}\cap I_{A_2}$ that contradicts with Corollary~\ref{di24}.
\end{proof}

\begin{theorem}\label{dia3} We have $\diam\br{O(S)}\ge3$, the strict inequality being equivalent to the existence of an ideal $\Ann(a)\nl R$
\ter{$a\in Z_R$} without zero-divisors.
\end{theorem}

\begin{proof} Similarly with examples from \cite{GM1,GM2} giving lower estimates of the diameter, for an arbitrary $a\in Z_R$, set $I:=\Ann(a)\nl R$
and $A:=J_n+aE_{n1}\in S$. Note that
\begin{itemize}
\item $A,A^T\in Z_S^*$, $O_S(A)=I^*E_{1n}$, $O_S(A^T)=I^*E_{n1}$\~
\item $a_{12}=1\ne a_{21}$, $(AA^T)_{11}=1$ and $O_S(A)\cap O_S(A^T)=\es$, that implies $d(A,A^T)\ge3$\~
\item if $I^*I^*\notni0$, then $\br{O_S(A)}\br{O_S(A^T)}=(I^*I^*)E_{11}\notni0$ and, hence, $d(A,A^T)\ge4$.
\end{itemize}
Due to mentioned above, $\diam\br{O(S)}\ge3$, the strict inequality following from the existence of an ideal $\Ann(a)\nl R$ \ter{$a\in Z_R$} without
zero-divisors. Conversely, in the case of the strict inequality, by Lemma~\ref{ieq}, for some elements $A\in Z_S$ and $a:=\det A\in Z_R$, the ideal
$I_A=\Ann(a)\nl R$ does not have zero-divisors.
\end{proof}

Now the main Theorem~\ref{main} follows from Theorems \ref{dia4} and~\ref{dia3}, and Corollary~\ref{di24}. It implies (see~\eqref{radi}) the statements
\ref{rad4} and~\ref{rad3} of Theorem~\ref{submain}. Let us prove~\ref{rad2}.

Suppose that $\rad\br{O(S)}=2$. There exist elements $C\in Z_S^*$, $c\in R^*$ and $k,l\in Q_n$ such that $d(C,A)\le2$ ($A\in Z_S^*$) and $c_{kl}=c$. Further,
there exists a~permutation $\si\in S_n$ such that $m:=\si(k)\ne l$.

Let $a\in Z_R$ be an arbitrary element.

Set $I:=\Ann(a)\nl R$ and $A:=\Br{\sums{i\ne k}E_{i,\si(i)}}+aE_{km}\in S$. Note that
\begin{itemize}
\item $A\in Z_S^*$, $O_S(A)=I^*E_{mk}$\~
\item $A\ne C$ (otherwise $a_{kl}=c\ne0$, $m=l$)\~
\item $(m,k)\ne(k,l)$ (otherwise $m=k=l$), that implies $C\notin O_S(A)$.
\end{itemize}
Thus, $d(C,A)=2$, so, there exists an element $B\in Z_S^*$ orthogonal to $C$ and~$A$. We have $B=bE_{mk}$ where $b\in I^*$. Meanwhile, $BC=0$, $0=(BC)_{ml}=bc$,
$b\in\Ann(c)\cap I^*$.

Due to arbitrariness of $a\in Z_R$, the element $c\in R^*$ satisfies~\eqref{cond}.

Conversely, assume that \eqref{cond} holds for some $c\in R^*$. Show that the element $C:=cE\bw\in S^*$ satisfies, for each $A\in Z_S^*$, the inequality $d(C,A)\le2$.

Let $A\in Z_S^*$ be an arbitrary element. Then $\det A\in Z_R$, and, by~\eqref{cond}, there exists an element $b\in I_A^*$ such that $cb=0$. Further, according to
Corollary~\ref{orde}, there exists an element $B\in(bS)^*$ orthogonal to~$A$\~ in this case, $cB\in cbS=0$, $C\in Z_S^*$, $C\perp B\perp A$, $d(C,A)\le2$.

So, Theorem~\ref{submain} is completely proved.

\section*{Acknowledgements}

The author is grateful to Prof. \fbox{E.\,B.\?Vinberg} for exciting interest to algebra.

The author dedicates the article to E.\,N.\?Troshina.

\end{document}